\documentclass[12pt]{article}
\usepackage[pdfusetitle]{hyperref}
\hypersetup{
	linktocpage,
	breaklinks=true,   
	colorlinks=true,   
}
\usepackage[no-math]{fontspec}
\usepackage{longtable}

\usepackage[
backend=biber,
sorting=ynt,
maxcitenames=99,
maxbibnames=99
]{biblatex}
\usepackage[a4paper, margin=2cm]{geometry} 

\usepackage{graphicx} 
\usepackage[nodisplayskipstretch]{setspace}

\usepackage{booktabs} 
\usepackage{array} 
\usepackage{paralist} 
\usepackage{verbatim} 
\usepackage{subfig} 

\usepackage{tikz}
\usetikzlibrary{arrows.meta, decorations.markings}

\usepackage{fancyhdr} 
\usepackage[nottoc,notlof,notlot]{tocbibind} 
\usepackage[titles,subfigure]{tocloft} 
\usepackage{amsmath,amssymb,amsthm}
\usepackage{xcolor}
\newtheorem{theorem}{Theorem}

\newtheorem{lemma}[theorem]{Lemma}

\title{Computing the divisor summatory function in \texorpdfstring{$N^{9/28+o(1)}$}{N to the power 9/28 + o(1)}}
\author{Darren Li\thanks{Qiuzhen College, Tsinghua University}}
\date{September 28, 2026}
\begin{document}

\maketitle

\begin{abstract}
  We present an analytic algorithm to compute isolated values of the divisor summatory function in $N^{9/28+o(1)}$ expected time, improving upon existing methods that require $O(N^{1/3})$ time. Our method significantly departs from elementary approaches to analyze the geometry of the Dirichlet hyperbola, relying instead on Voronoï summation and accelerating the evaluation of the resulting exponential sum.
\end{abstract}

\section{Introduction}

Consider the divisor summatory function, or the number of lattice points under the hyperbola of Dirichlet's divisor problem:
\[D(N)=\sum_{n=1}^N d(n) = \sum_{xy \le N} 1 = \sum_{x=1}^N \left\lfloor \frac{N}{x}\right\rfloor = 2\sum_{x=1}^{\lfloor \sqrt N\rfloor} \left\lfloor \frac{N}{x}\right\rfloor - \lfloor \sqrt N\rfloor^2.\]
Here $d(n)$ counts positive divisors and $x,y$ range over positive integers.

The first $N^{1/3+o(1)}$-time algorithm to compute $D(N)$ is due to D.H.J. Polymath \cite[Section 2.1]{polymath2012deterministicmethodsprimes}; Sladkey \cite{sladkey2012successiveapproximationalgorithmcomputing} gives a distinct $O(N^{1/3})$-time, $O(\log n)$-space algorithm. Both algorithms obtain their speedup by replacing column-by-column summation with local rational-linear approximations to the hyperbola, allowing the contribution of large blocks of lattice points to be evaluated in bulk. The Polymath algorithm does this by Diophantine approximation on dyadic intervals, while Sladkey uses a recursive decomposition by rational tangent lines and unimodular changes of coordinates.

There is circumstantial evidence to believe that the exponent $1/3$ for local linear approximation of the hyperbola in the spirit of Voronoï's classical polygonal treatment \cite{Voronoi1903} of the Dirichlet divisor problem is characteristic of this class of methods. Over a relative interval of width $u$, the area between $xy=N$ and a tangent line is of order $Nu^3$, so the lattice scale is reached when $u \asymp N^{-1/3}$; along these lines Alcántara, Blanco, Criado, and Santos in \cite{alcántara2025convexhullintegerpoints} prove an $\Omega(N^{1/3})$ lower bound on the size of the convex hull of the lattice points above the hyperbola. More locally, Balog and Bárány in \cite{balog2026integerhullsetxyin} identifies $1/3$ as the scale of the strip
\[N\leq xy\leq N+2N^{1/3}\]
in which all convex hull points lie, and therefore a line therein can span only a relative \(x\)-interval of length $1+O(N^{-1/3})$. Consequently, even on a single fixed multiplicative interval $X\leq x\leq 2X$, an approximation which resolves the hyperbola to this precision requires $\Omega(N^{1/3})$ linear pieces. 
Although this does not give a computational lower bound for $D(N)$, it suggests that improving on the exponent $1/3$ requires a mechanism beyond a straightforward local piecewise-linear resolution of the hyperbola.

We circumvent this limitation by a completely analytic approach, dropping the hyperbola in favor of summing the divisor sequence directly. After applying Voronoï summation to a smoothed count, the dual length that arises at the dominant scale is $M=N^{5/14+o(1)}$. To calculate this fast, we decompose into blocks and reuse a table of short polynomial-phase sums. We first explain the reuse mechanism with an unrestricted coefficient table and then present optimizations to reduce the complexity from $N^{10/31+o(1)}$ to $N^{9/28+o(1)}$.
The smoothing correction requires divisor counts for the $O(H)$ integers with $|n-N|\le H$, where $H=N^{9/28+o(1)}$. The only randomization is in their Lenstra-Pomerance factorization \cite{lenstraPomerance1992}.

\begin{theorem}
  There is an unconditional randomized algorithm which, for every positive
integer \(N\), returns \(D(N)\) exactly and terminates with probability one in
\[\text{space}\ O\left(N^{9/28+o(1)}\right)\qquad\text{and}\qquad\text{expected time}\ O\left(N^{9/28+o(1)}\right)\]
measured in bits of storage and bit operations, respectively. The randomization is isolated to the factorization of $N^{9/28+o(1)}$ integers: if their divisor counts are supplied, the remaining computation is deterministic.
\end{theorem}

\textbf{AI disclosure.} The core contribution is due to LLMs. We simplified the resulting proof.

\section{Outline of the proof}

The smoothing width $H$ governs two competing costs. Correcting the smoothed count requires $O(H)$ divisor values, while Voronoï summation gives a dual sum of length $M=N/H^2$ up to logarithmic factors. We first obtain this sum, then explain how a table shared by many short blocks makes its evaluation faster than summing its terms individually.

The celebrated result of Voronoï for the divisor summatory function \cite[(3.1), p.~47]{ivic1983} is
\begin{align}
  D_*(u)={\sum_{n\le u}}'d(n)=A(u)-\sum_{m\ge 1}d(m)\sqrt{\frac um}\left[ Y_1(4\pi\sqrt{mu})+\frac 2\pi K_1(4\pi\sqrt{mu}) \right],
  \label{eq:voronoiDivisorsum}
\end{align}
where $A(u)=u(2\gamma-1+\log u)+\frac 14$, the prime assigns half weight to $n=u$ when $u$ is an integer, and $Y_1$ and $K_1$ denote the Bessel and modified Bessel functions of the second kind of order 1. The series is boundedly convergent on each compact subinterval of $(0,\infty)$.

Fix a suitable probability density $\eta(t)$ smooth on $\mathbb R$ and supported on $(-1, 1)$, e.g.
\[
  \beta(t)=\begin{cases}\exp\left(-\frac{1}{1-t^2}\right),&|t|<1,\\0,&|t|\ge1,\end{cases}
  \quad Z=\int_{-1}^1\beta(t)dt,\quad \eta(t)=Z^{-1}\beta(t).
\]
Lemma \ref{lem:bump} gives a Gevrey-2 bound for $\eta$; its evaluation is discussed at the end of Section \ref{sec:estimates}.

At this point we fix the scale of the smoothing $1\le H\le\min(\sqrt N,N/3)$; we will later choose it to be $N^{9/28+o(1)}$. Write square root coordinates
\[u(t)=\left(\sqrt N + Lt\right)^2, \quad L = \frac{H}{4\sqrt{N}}, \quad |t|<1,\]
where $L$ is the scale of the smoothing radius in these coordinates. This choice makes the oscillatory phase $4\pi\sqrt{mu}$ linear in $t$, so that averaging will introduce a Fourier transform. The smoothed object is
\[S=\int_{-1}^1 \eta(t)D_*(u(t))dt = \sum_{n\ge 1}d(n)v(n)\]
where the weights $v(n)$ are
\[v(n)=\int_{-1}^1 \eta(t)\mathbf{1}_{\{n\le u(t)\}}dt=\begin{cases}
  1 & \sqrt{n} \le \sqrt{N}-L \\
  \int_{t_n}^1 \eta(t)dt & |\sqrt{n}-\sqrt{N}| < L\\
  0 & \sqrt{n} \ge \sqrt{N}+L
\end{cases}, \quad t_n = \frac{\sqrt n-\sqrt N}L.\]

It follows that the correction $D(N)-S$ is a finite sum with $O(H)$ terms, each of the form $d(n)(\mathbf1_{\{n\le N\}}-v(n))$. Lemma \ref{lem:smoothing} gives its evaluation using the factorization of $O(H)$ integers. The remaining object of our investigation is $S$.

By \eqref{eq:voronoiDivisorsum},
\begin{align}
  S=\int_{-1}^1 \eta(t)A(u)dt - \sum_{m\ge 1} d(m)\int_{-1}^1 \eta(t)\sqrt{\frac um}\left[ Y_1(4\pi \sqrt{mu})+\frac{2}{\pi}K_1(4\pi \sqrt{mu}) \right]dt. \label{eq:Sform1}
\end{align}
Here and below $u=u(t)$. Denote the first integral by $J$; Section \ref{sec:estimates} gives its evaluation to arbitrary precision in polylogarithmic time.

Equation \eqref{eq:K-error} bounds the contribution of all $K_1$ terms by $4N^{1/4}e^{-10\sqrt N}\le N^{-10}$. Lemma \ref{lem:bessel} shows that retaining just the leading term of the Hankel asymptotic expansion,
\[-\sqrt{\frac um} Y_1(4\pi \sqrt{mu}) = \frac{u^{1/4}}{\pi \sqrt{2} m^{3/4}}\cos \left( 4\pi\sqrt{mu}-\frac \pi4\right) + R(m,u),\]
the contribution of all remainder $R$ terms to $S$ is $O(N^{-1/4})$.

Write $e(t)=\exp(2\pi it)$, so that we can expand $\sqrt u=\sqrt N+Lt$ to get
\begin{align*}
  \eta(t)\frac{u^{1/4}}{\pi \sqrt{2} m^{3/4}}\cos \left( 4\pi\sqrt{mu}-\frac \pi4\right)&=
  \Re \left[ \frac{(1-i)N^{1/4}}{2\pi  m^{3/4}}e(2\sqrt{Nm})\  \eta(t)\sqrt{1+LtN^{-1/2}}e(2L\sqrt{m}t) \right].
\end{align*}
Define
\begin{align*}
  \delta=\frac{L}{\sqrt N}=\frac{H}{4N},\quad F_\delta(\xi) = \int_{-1}^1 \eta(t)(1+\delta t)^{1/2}e(\xi t)dt,\quad W(m)=\frac{(1-i)N^{1/4}}{2\pi m^{3/4}}F_\delta(2L\sqrt m).
\end{align*}
Lemma \ref{lem:fourier} gives the explicit decay
$|F_\delta(\xi)|\le16e^{-\sqrt{|\xi|}/16}$ for real $\xi$.
For $N\ge2^{20}$, Lemma \ref{lem:truncation} gives a fixed choice for $M=\widetilde O(N/H^2)$ so that the total discarded tail is at most $N^{-10}$. Putting these together, \eqref{eq:Sform1} becomes
\begin{align}
  S=J+\Re \sum_{m=1}^M d(m)W(m)e(2\sqrt{Nm})+E_{\rm analytic}, \label{eq:Sform2}
\end{align}
where $|E_{\rm analytic}|\to 0$ as $N\to \infty$.

Consider the sum of \eqref{eq:Sform2}. Rearranging to eliminate the divisor term,
\begin{align}
  \sum_{m=1}^M d(m)W(m)e(2\sqrt{Nm})=\sum_{a\le\sqrt M}\left[2\sum_{b=a}^{\lfloor M/a\rfloor}W(ab)e(2\sqrt{Nab})-W(a^2)e(2a\sqrt N)\right].
  \label{eq:SformTargetSum}
\end{align}
To evaluate the inner sum fast, we split it into blocks of length $K$. Write the desired sum as
\[S(a,B,K)=\sum_{k=0}^{K-1}W(a(B+k))e(2\sqrt{Na(B+k)}).\]

The blocks have different phases and weights. To reuse a precomputation among them, we separate the phase into an integer polynomial and a remainder which can be absorbed into the weight. Define
\[t=t(k)=\frac{k}{K},\qquad g(t)=W(a(B+Kt)),\qquad \phi(t)=2\sqrt{NaB}\left( 1+\frac KB t \right)^{1/2}.\]
The intuition is to decompose $\phi$ into a degree-$R$ polynomial with \emph{integer} coefficients, $R$ globally fixed, plus an analytic remainder $V(t)$, i.e.
\[\phi(t) = V(t) + \sum_{j=0}^R q_j t^j, \quad e(\phi(t))=e\left( \sum_{j=0}^R q_jt^j \right)e(V(t)).\]
For instance, take $q_j$ to be a nearest integer to the corresponding Taylor coefficient of $\phi$.

For sufficiently short blocks, $g(t)e(V(t))$ can be approximated by a polynomial $h$ of polylogarithmic degree, with complex coefficients which can be calculated in polylogarithmic time. Theorem \ref{thm:block} establishes this approximation:
\[S(a,B,K)\approx\sum_{k=0}^{K-1} e\left( \sum_{j=0}^R q_jt^j \right) \sum_{r=0}^{\deg h} h_r t^r =  \sum_{r=0}^{\deg h} h_r \left[\sum_{k=0}^{K-1} t^r e\left( \sum_{j=0}^R q_jt^j \right)\right];\]
the error is $K \|g e(V)-h\|_\infty$.

Crucially, the bracketed moment depends on $q_j$ only through its remainder mod $K^j$, because adding $K^j$ changes the phase at $t=k/K$ by the integer $k^j$. Thus all blocks of the same length draw their moments from a single finite table, regardless of $a$ and $B$. Denote the moment by $T_r(K;q_1,\dots,q_R)$. For a fixed $K$, there are $K^{R(R+1)/2}$ keys, and once this table is available each block costs only polylogarithmic time. Section \ref{sec:sec5} constructs the table in time essentially linear in its number of entries.

The degree $R$ determines how long these blocks may be. The first omitted Taylor coefficient has size
\[|[t^{R+1}]\phi(t)| \asymp \sqrt{NaB}\left( \frac{K}{B} \right)^{R+1},\]
so, with $K/B\le1/8$, keeping the Taylor tail bounded on a fixed complex disk requires the scale
\begin{align}K(a,B)^{2R+2} \lesssim \frac{B^{2R+1}}{Na}.\label{eq:curvatureRestriction}\end{align}
We call this the curvature restriction. Under this condition, the polynomial approximation follows from the holomorphic bounds in Section \ref{sec:estimates}. There is a tradeoff: larger $R$ permits longer blocks but increases the table size. Lemma \ref{lem:block-count} shows that degree three forces too many blocks to improve on $1/3$; degree four is sufficient for the algorithm below.

For $R=4$ the unrestricted table has $K^{10}$ entries. Balancing its construction with $M/K$ block queries, the preliminary complexity of evaluating the sum of \eqref{eq:Sform2} is
$\widetilde O(M^{10/11})$, and balancing against the smoothing cost suggests $N^{10/31+o(1)}$ total time. For the complete argument accounting for the curvature restriction see \eqref{eq:unrestricted-cost}.

The unrestricted table pays for every possible fourth residue. To improve this further, we split the summands of \eqref{eq:SformTargetSum} into dyadic product ranges (``shells'') $X\le ab<2X$. In such a shell the coefficient of $k^4$ in the phase at $b=B$ is
\[-\frac5{64}\sqrt{Na}B^{-7/2}= O(\sqrt N X^{-3/2}),\]
so after scaling $k=Kt$ and rounding, the fourth coordinate within a shell can only take $O(1+C_XK^4)$ values, where $C_X\asymp\min(1,\sqrt N X^{-3/2})$. Section \ref{sec:sec6} balances this against the $\widetilde O(X/K)$ block queries by taking $K\asymp(X/C_X)^{1/11}$, again subject to the curvature restriction. The dominant contribution is $\widetilde O(N^{1/22}M^{17/22})$; we also bound the shorter blocks and the smaller product ranges. Balancing against the smoothing cost $H$, with $M=N/H^2$ up to logarithmic factors, gives $H=N^{9/28+o(1)}$.

At the dominant scale, suppressing logarithmic factors, the resulting balance is
\[H=N^{9/28},\quad M=N^{5/14},\quad K=N^{1/28}.\]
Here $C_M=N^{-1/28}$, so the fourth coordinate has only $O(K^3)$ possible values. The table has size $K\cdot K^2\cdot K^3\cdot K^3=K^9$, matching the $M/K=N^{9/28}$ block queries and the smoothing cost. This saving in the fourth coordinate is the refinement that gives the final exponent.

\section{Estimates}
\label{sec:estimates}

The reduction uses two properties of $W$: decay at real frequencies, which permits truncation at $M=N/H^2$ up to logarithmic factors, and controlled growth on a complex neighborhood, which permits polynomial approximation on each block. We establish these bounds and the evaluation estimates needed to make the approximations constructive.

\paragraph{The finite dual sum.}
The smoothing changes the count only near $N$. Indeed, with $\delta=H/4N\le1/12$,
\[u=N(1+\delta t)^2,\quad |u-N|\le \frac{H}{2}+\frac{H^2}{16N}<H,\]
so $v(n)=\mathbf1_{\{n\le N\}}$ outside $|n-N|\le H$ and
\[D(N)-S=\sum_{|n-N|\le H}d(n) \left(\mathbf1_{\{n\le N\}}-v(n)\right).\]
The half weights in $D_*$ affect only finitely many points of the integral defining $S$; bounded convergence of the Voronoï series on $u([-1,1])$ justifies its termwise integration in \eqref{eq:Sform1}.

We can replace the Bessel kernel by its leading oscillatory term before truncating the series. The reason is the summable factor $m^{-5/4}$ in the remainder.

\begin{lemma}\label{lem:bessel}
For $m,u>0$, write
\[-\sqrt{\frac um}Y_1(4\pi\sqrt{mu})=\frac{u^{1/4}}{\pi\sqrt2m^{3/4}}\cos(4\pi\sqrt{mu}-\pi/4)+R(m,u).\]
Then
\[
  |R(m,u)|\le\frac{3}{32\sqrt2\pi^2}u^{-1/4}m^{-5/4},
  \quad
  \sum_{m\ge1}d(m)\int_{-1}^1\eta(t)|R(m,u)|dt<\frac14N^{-1/4}.
\]
\end{lemma}

\begin{proof}
Taking imaginary parts in the Hankel estimate \cite[10.17(iii)]{dlmf}
\[H_1^{(1)}(z)=\sqrt{\frac2{\pi z}}e^{i(z-3\pi/4)}(1+\epsilon(z)),\quad |\epsilon(z)|\le\frac3{8z}\quad(z>0),\]
and substituting $z=4\pi\sqrt{mu}$ gives the pointwise bound. With $u\ge N(1-\delta)^2$ and $\int\eta=1$, its total contribution is at most
\[\frac{3\zeta(5/4)^2}{32\sqrt2\pi^2\sqrt{1-\delta}}N^{-1/4}<\frac14N^{-1/4},\]
using $\sum d(m)m^{-5/4}=\zeta(5/4)^2$ and $\zeta(5/4)\le5$.
\end{proof}

The $K_1$ terms are exponentially small. From \cite[10.40(ii)]{dlmf}
\[0<K_1(z)\le\sqrt{\frac\pi{2z}}e^{-z}\left(1+\frac3{8z}\right)\quad(z>0),\]
we obtain, using $4\pi\sqrt{mu}>11\sqrt{Nm}$ and $d(m)\le2\sqrt m$,
\begin{align}
  E_K&\stackrel{\rm def}{=}\frac2\pi\sum_{m\ge1}d(m)\int_{-1}^1\eta(t)\sqrt{\frac um}K_1(4\pi\sqrt{mu})dt\notag\\
  &\le N^{1/4}\sum_{m\ge1}m^{-1/4}e^{-11\sqrt{Nm}}\le4N^{1/4}e^{-10\sqrt N}\le N^{-10}. \label{eq:K-error}
\end{align}
Here the sum is bounded by the integral comparison
\[\sum_{m\ge1}m^{-1/4}e^{-a\sqrt m}\le e^{-a}(1+2/a+2/a^2),\]
and the last inequality follows from $e^{-x}\le21!x^{-21}$ and $4\cdot21!<10^{21}$.

To truncate the remaining oscillatory series, we use the flatness of the smoothing density at its endpoints. The following derivative bound allows repeated integration by parts with the number of integrations chosen in terms of the frequency.

\begin{lemma}\label{lem:bump}
The density $\eta$ is smooth on $\mathbb R$, with all derivatives vanishing at $\pm1$, and for $r\ge0$
\[\|\eta^{(r)}\|_\infty\le4\cdot36^r(r!)^2.\]
\end{lemma}

\begin{proof}
Put $\psi(x)=e^{-1/x}$ for $x>0$ and $\psi(x)=0$ otherwise. On $|z-x|=x/2$, the bound $\Re(1/z)\ge2/(9x)$ and Cauchy's estimate give
\[|\psi^{(r)}(x)|\le r!(2/x)^r e^{-2/(9x)}\le9^r(r!)^2,\]
where the second inequality follows by maximizing in $x$ and using $(r/e)^r\le r!$ (the case $r=0$ is immediate). The first bound tends to zero as $x\to0^+$ for each fixed $r$, proving smoothness of the zero extension.

Now $\beta(t)=\psi(2(1-t))\psi(2(1+t))$ on the reals, so Leibniz's rule yields
\[\|\beta^{(r)}\|_\infty\le\sum_{k=0}^r\binom rk\left(2^k\|\psi^{(k)}\|_\infty\right)\left(2^{r-k}\|\psi^{(r-k)}\|_\infty\right)\le36^r(r!)^2.\]
Dividing by $Z>1/4$, as $\beta(t)>1/4$ on $[-1/2,1/2]$, proves the assertion.
\end{proof}

\begin{lemma}\label{lem:fourier}
For $0\le\delta\le1/12$ and real $\xi$,
\[|F_\delta(\xi)|\le16e^{-\sqrt{|\xi|}/16}.\]
\end{lemma}

\begin{proof}
The zero extension of $f_\delta(t)=\eta(t)(1+\delta t)^{1/2}$ is smooth. Combining the preceding lemma with
\[\left\|\left((1+\delta t)^{1/2}\right)^{(s)}\right\|_\infty\le2s!\]
in Leibniz's rule, and using the support interval $[-1,1]$, gives
\begin{equation}
  \|f_\delta^{(r)}\|_\infty\le8\cdot64^r(r!)^2,\quad \|f_\delta^{(r)}\|_1\le16\cdot64^r(r!)^2.
  \label{eq:fdelta-gevrey}
\end{equation}
There are no boundary terms in integration by parts, hence for $\xi\ne0$
\[|F_\delta(\xi)|\le16\left(\frac{64}{2\pi|\xi|}\right)^r(r!)^2.\]
To balance the frequency against the factorial growth, put $x=\sqrt{\pi|\xi|/128}$ and take $r=\lfloor x\rfloor$ when $x\ge2$. Since $x/2\le r\le x$ and $r!\le r^r$, the last bound is at most
\[16\cdot4^{-r}\le16e^{-(\log2)\sqrt{\pi|\xi|/128}}\le16e^{-\sqrt{|\xi|}/16}.\]
For $x<2$ we have $|\xi|<256$, so the claimed upper bound exceeds $16/e$, whereas $|F_\delta(\xi)|\le\int\eta(t)\sqrt{1+\delta t}\,dt\le\sqrt{13/12}$.
\end{proof}

The square-root argument of $F_\delta$ turns this into fourth-root exponential decay in $m$. This accounts for the fourth power of the logarithm in the cutoff.

\begin{lemma}\label{lem:truncation}
For $N\ge2^{20}$, let
\[M_*=\frac{2^{36}N\lceil\log_2 N\rceil^4}{H^2}.\]
Then the tail after $M=\lfloor M_*\rfloor$ satisfies $\sum_{m>M}d(m)|W(m)|\le N^{-10}$.
\end{lemma}

\begin{proof}
Put
\[\lambda=\frac{\sqrt H}{32N^{1/4}},\quad Y=\lambda M_*^{1/4}=16\lceil\log_2 N\rceil.\]
At the dual frequency, Lemma \ref{lem:fourier} gives $|F_\delta(2L\sqrt m)|\le16e^{-\lambda m^{1/4}}$; with $|\tfrac{1-i}{2\pi}|<1/4$ and $d(m)\le2\sqrt m$, this bounds the tail by
\[\sum_{m>M}d(m)|W(m)|\le8N^{1/4}\sum_{m>M}m^{-1/4}e^{-\lambda m^{1/4}}.\]
The summand is decreasing and the first omitted integer exceeds $M_*$. Bounding the sum by its first term and an integral, with the substitution $y=\lambda m^{1/4}$, gives
\[\sum_{m>M}d(m)|W(m)|\le8N^{1/4}\left[M_*^{-1/4}+4\lambda^{-3}(Y^2+2Y+2)\right]e^{-Y}\le2^{21}N(Y+2)^2 2^{-Y},\]
where $H\ge1$ and $M_*\ge1$ suffice for the last inequality. As $16\log_2N\le Y\le32\log_2N-2$, this is at most $2^{31}(\log_2N)^2N^{-15}<N^{-10}$, as desired.
\end{proof}

We have now reduced the averaged Bessel series to the finite sum in \eqref{eq:Sform2}. Its analytic error is the sum of the absolutely summable Hankel remainder, the $K_1$ contribution, and the Fourier tail: for $N\ge2^{20}$,
\[|E_{\rm analytic}|\le\frac14N^{-1/4}+2N^{-10}<\frac1{100}.\]

\paragraph{Evaluation of the weights and correction.}
By a classical quadrature argument, the samples of $W$, the main integral $J$, and each weight in the smoothing correction can be evaluated to error $2^{-p}$ in $(\log N+p)^{O(1)}$ bit operations.

Specifically, for $f\in C^\infty([-1,1];\mathbb C)$ with $\|f^{(r)}\|_\infty\le AB^r(r!)^2$ and $A,B\ge1$, use composite interpolatory quadrature \cite[\S3.5(iv)]{dlmf}: take $s=\lceil p+\log_2A+8\rceil$ equally spaced nodes on each of $\lceil8Bs\rceil$ equal subintervals. The Lagrange remainder \cite[(3.3.5)]{dlmf}, applied to real and imaginary parts, bounds the total integral error by
\[4A(Bh)^s s!\le4A4^{-s}<2^{-p},\quad h=2/\lceil8Bs\rceil.\]
Thus $O(B(p+\log_2A+1)^2)$ samples suffice, with rational nodes and weights. The absolute weights sum to at most $2\cdot6^s$, since the absolute Lagrange basis functions on each subinterval sum to at most $(2s-2)^{s-1}/(s-1)!\le6^s$; sample accuracy of $p+O(s)$ bits therefore suffices, and construction and evaluation take time polynomial in $B+p+\log A$ given the samples.

Applied to $\beta$, the rule computes $Z$, and hence $\eta$ and its partial integrals, in polynomial time. Near $\pm1$ we may set $\beta$ to zero when $1/(1-t^2)\ge p+2$, and elsewhere the exponential just has argument $O(p)$.

For $|\xi|\le Q$, the same construction applies to $F_\delta(\xi)$, using
\[\|(f_\delta(t)e(\xi t))^{(r)}\|_\infty\le8(64+2\pi Q)^r(r!)^2\]
from \eqref{eq:fdelta-gevrey}. Its cost is polynomial in $Q+p$, with input errors controlled by $|\partial_\delta F_\delta|\le1$ and $|\partial_\xi F_\delta|\le4\pi$. Throughout $1\le m\le3M_*/2$ we have $2L\sqrt m=O((\log N)^2)$, so the required samples of $W$ have the asserted cost. This also makes the polynomial in Lemma \ref{lem:weight} constructive: use real Chebyshev samples with $O(d+\log\mathcal B+p)$ accuracy bits, including the error in the sample nodes, which is controlled by Cauchy's estimate on the ellipse.

For the main integral, writing $\log u=\log N+2\log(1+\delta t)$ gives
\[\|(\eta(t)A(u))^{(r)}\|_\infty\le CN\log N C_1^r(r!)^2\]
with absolute constants $C,C_1$. The prefactor $N\log N$ costs only $O(\log N)$ additional precision bits, so $J$ is covered by the same quadrature estimate.

\begin{lemma}\label{lem:smoothing}
Given the divisor counts for $|n-N|\le H$, the smoothing correction $D(N)-S$ can be computed to error $2^{-p}$ in $(\log N+p)^{O(1)}$ time per term.
\end{lemma}

\begin{proof}
Since, say, $d(n)\le3\sqrt N$ on the correction interval, it suffices to evaluate the partial integrals $v(n)$ to error $2^{-p}/(16HN)$ each. Their constant extensions beyond $[-1,1]$ are Lipschitz, allowing us to round $t_n$ as necessary.
\end{proof}

Factoring the $O(H)$ integers with $|n-N|\le H$ supplies their divisor counts in $HN^{o(1)}$ expected bit operations by Lenstra-Pomerance \cite{lenstraPomerance1992}, using $N^{o(1)}$ auxiliary space when the integers are processed sequentially.

We remark that the scale this problem requires is just shy of the current best bound for segmented sieve algorithms, due to Helfgott \cite{helfgott2019improvedsieveeratosthenes}, of sieving intervals of length $N^{1/3}$ in essentially linear time, and reducing $1/3$ to $9/28$ would suffice to replace Lenstra-Pomerance and make our algorithm deterministic. This appears to be immensely difficult; we have made no progress on this problem. The exponent $1/3$ manifests for the same reason as those listed in the introduction, linear approximation of the Dirichlet hyperbola, and indeed \cite[Section 5.1]{helfgott2019improvedsieveeratosthenes} suggests that improving sieving would require a genuinely deeper understanding of the approximation of the hyperbola, the same obstruction that we sought to circumvent. 

\paragraph{Polynomial approximation on a block.}
This is the critical property of the weights that our analytic setup gives us so that we can have Theorem \ref{thm:block}. For the block algorithm we need control of $W$ off the real axis. The defining integral of $F_\delta$ is entire, with
\begin{equation}
  |F_\delta(z)|\le\sqrt{13/12}e^{2\pi|\Im z|}\le2e^{2\pi|\Im z|}.
  \label{eq:fourier-complex}
\end{equation}
Although this bound grows exponentially, the frequencies on our blocks are only $O((\log N)^2)$, since $H\sqrt{aB/N}=O((\log N)^2)$ for $aB\le M$. This is small enough to give a polynomial approximation of polylogarithmic degree.

Rigorously speaking, so that we have a finite representation and the precision needed in Section \ref{sec:sec4} is polylogarithmic, we also need to bound the bit length of each coefficient and a total coefficient norm, though that is comparatively easy.

\begin{lemma}\label{lem:weight}
Let $a,B,K$ be positive integers with $aB\le M$ and $K\le B/2$. Write $\|P\|_{\rm coef}=\sum_j|[t^j]P|$. For every integer $p\ge1$ there is a polynomial $P$ with dyadic complex coefficients such that
\[\sup_{0\le t\le1}|W(a(B+Kt))-P(t)|\le2^{-p},\]
and $\deg P$, $\|P\|_{\rm coef}$, as well as the bit length of each coefficient, is $O((\log N)^2+p)$.
\end{lemma}

\begin{proof}
Let $g(t)=W(a(B+Kt))$. On the closed Bernstein ellipse $\mathcal E_2$ for $[0,1]$, bounded by
\[t=\frac12+\frac14(z+z^{-1}),\quad |z|=2,\]
the variable $v=B+Kt$ satisfies
\[\Re v\ge15B/16,\quad |\Im v|\le3K/8,\quad |\Im\sqrt{av}|=\frac{a|\Im v|}{2\Re\sqrt{av}}\le\frac14\sqrt{aB}\frac KB.\]
In particular, $v$ stays in the right half-plane, away from the branch cut of the principal fractional powers in $g$. The function $g$ is therefore holomorphic on a neighborhood of $\mathcal E_2$, and \eqref{eq:fourier-complex} gives
\[\sup_{\mathcal E_2}|g|\le N^{1/4}(aB)^{-3/4}\exp\left(\frac\pi4H\sqrt{\frac{aB}N}\frac KB\right)\le\exp(C(\log N)^2).\]
Call the last bound $\mathcal B$. The degree-$d$ Chebyshev interpolant has error at most $4\mathcal B2^{-d}$ \cite[Theorems 8.1-8.2]{trefethen2013}, so $d=O((\log N)^2+p)$ suffices to make this at most $2^{-p-1}$.

It remains to bound the bit length after conversion to monomials. The recurrence
\[T_{j+1}(2t-1)=(4t-2)T_j(2t-1)-T_{j-1}(2t-1)\]
bounds the coefficient norm of $T_j(2t-1)$ by $7^j$, by induction, and hence that of the interpolant by $4\mathcal B(d+1)7^d$. Put $q=p+\lceil\log_2(d+1)\rceil+2$ and round the real and imaginary parts of each coefficient to multiples of $2^{-q}$. Each coefficient changes by at most $2^{-q}$, adding at most $(d+1)2^{-q}\le2^{-p-2}$ to the error on $[0,1]$. The resulting coefficients have the form $2^{-q}(u+iv)$ with integers $u,v$ of bit length $O(q+\log\mathcal B+d)=O((\log N)^2+p)$, as desired.
\end{proof}

We now have all of the requisite estimates for the correctness of our analytic setup.

\section{Computing \texorpdfstring{$S(a,B,K)$}{S(a,B,K)}}
\label{sec:sec4}

On a sufficiently short block, the part of the phase beyond degree four can be absorbed into the weight without spoiling its polynomial approximation. This reduces the block to a short linear combination of table entries. The length restriction matters, however: the saving depends on how many blocks are needed to cover the rows.

Call $(a,B,K)$ admissible if
\begin{equation}
  a\le B,\quad a(B+K-1)\le M,\quad 8K\le B,\quad 2^{20}NaK^{10}\le B^9.
  \label{eq:admissible-block}
\end{equation}
The last condition is the quartic curvature restriction, bounding the Taylor remainder on a fixed complex neighborhood of $[0,1]$ for us to use the approximation argument of Lemma \ref{lem:weight}.

\begin{theorem}\label{thm:block}
For $0<\varepsilon<1$, an admissible $S(a,B,K)$ can be evaluated to error $\varepsilon$ in $(\log N+\log(1/\varepsilon))^{O(1)}$ time, given the moment tables $T_0,T_1,\dots,T_d$ up to degree
\[d=O\left((\log N)^2+\log(1/\varepsilon)\right)\]
with $O((\log N)^2+\log(1/\varepsilon))$ bits of precision.
\end{theorem}

\begin{proof}
Write
\[\phi(t)=\sum_{j\ge0}\alpha_jt^j,\quad\alpha_j=2\sqrt{NaB}\binom{1/2}{j}(K/B)^j,\]
and round $\alpha_0,\ldots,\alpha_4$ to nearest integers $q_0,\ldots,q_4$. Put
\[V(t)=\phi(t)-\sum_{j=0}^4q_jt^j.\]
The rounding leaves bounded errors in the first five coefficients. For the higher coefficients, admissibility gives, on $|t|\le2$,
\[\left|\sum_{j\ge5}\alpha_jt^j\right|\le2\sqrt{NaB}\frac{(2K/B)^5}{1-2K/B}\le\frac{256}{3}\sqrt{\frac{NaK^{10}}{B^9}}\le\frac1{12},\]
using $|\binom{1/2}{j}|\le1$. Hence $|V(t)|\le31+1/12<32$ on this disk, and absorbing the residual phase $e(V(t))$ costs only an absolute factor in the complex bound for the weight -- although we can dismiss it so easily now this is the entire reason for the curvature restriction as otherwise $V(t)$ would be sigificantly worse than $O((\log N)^2)$ and promptly destroy the following bound.

Since $\mathcal E_2$ lies inside $|t|<2$, the product
\[f(t)=g(t)e(V(t)),\quad g(t)=W(a(B+Kt)),\]
is holomorphic near $\mathcal E_2$ and bounded there by $\exp(C(\log N)^2)$. The interpolation and coefficient rounding from Lemma \ref{lem:weight} give a polynomial $h$ with dyadic complex coefficients satisfying
\begin{equation}
  \|f-h\|_{\infty,[0,1]}\le\frac{\varepsilon}{2K},
  \quad
  \deg h+\log(1+\|h\|_{\rm coef})=O\left((\log N)^2+\log(1/\varepsilon)\right).
  \label{eq:block-polynomial}
\end{equation}
Its real samples are computed by evaluating $g$ as in Section \ref{sec:estimates} and subtracting $\sum_{j=0}^4q_jt^j$ from $\phi(t)$ to obtain $V$. Both $\phi$ and the $q_j$ have $O(\log N)$ integer bits, so this subtraction and the interpolation take polynomial time in the stated precision parameters.

At the sample points $t=k/K$, replacing $q_j$ in the polynomial phase by its residue $\bar q_j$ modulo $K^j$ leaves the exponential unchanged; $V$ remains as defined with the original integers. With $e(q_0)=1$, this gives
\begin{equation}
  S(a,B,K)=\sum_{r=0}^{\deg h}h_rT_r(K;\bar q_1,\bar q_2,\bar q_3,\bar q_4)+E,\quad |E|\le\varepsilon/2.
  \label{eq:block-contraction}
\end{equation}
Moment errors of at most $\varepsilon/[2(1+\|h\|_{\rm coef})]$ contribute less than $\varepsilon/2$ more. By \eqref{eq:block-polynomial}, the required precision is still $O((\log N)^2+\log(1/\varepsilon))$ bits, and the finite sum can be formed exactly from the dyadic approximations.
\end{proof}

To count the blocks, impose a length cap $U$ to limit the size of the tables. Restricting lengths to powers of two loses at most a factor two when rounding a permitted length down, and makes the total table cost over all lengths comparable to that of the largest. The following estimate separates the contribution of the cap from that of curvature, counting an individually evaluated summand as one unit of work.

\begin{lemma}\label{lem:block-count}
Fix $R\ge2$ and $c>0$. For $2\le M\le N$ and $1\le U\le M$, the rows $a\le b\le\lfloor M/a\rfloor$ can be partitioned into individual summands and blocks of dyadic lengths satisfying
\[K\le U,\quad 8K\le B,\quad NaK^{2R+2}\le cB^{2R+1},\]
with total work
\begin{equation}
  O_{R,c}\left(\frac{M\log(2M)}U+N^{\frac1{2R+2}}M^{\frac{R+2}{2R+2}}+N^{\frac1R}+\sqrt M\log(2M)\right).
  \label{eq:block-count}
\end{equation}
If $M\ge N^{1/R}$, every such partition has work
\begin{equation}
  \gg_{R,c}\max\left\{\frac MU,N^{\frac1{2R+2}}M^{\frac{R+2}{2R+2}}\right\}.
  \label{eq:curvature-lower}
\end{equation}
The lower bound holds without the requirement that lengths be dyadic.
\end{lemma}

\begin{proof}
Starting at $B=a$ in each row, take the largest permitted dyadic length that fits in the remaining row. If even length one is forbidden, evaluate that summand directly and advance $B$ by one. Such exceptions have $B<8$ or $B^{2R+1}<Na/c$; since $a\le B$, they lie in $B\ll_{R,c}N^{1/(2R)}$ and number $O_{R,c}(N^{1/R})$ in total.

A block limited by the remaining row length removes more than half of that remainder, giving $O(\log(2M))$ final fragments per row. On every other block, dyadic rounding loses at most a factor two and $B\le b\le9B/8$. We can therefore count these blocks by integrating the reciprocals of the three length bounds. The cap $U$ contributes $O(M\log(2M)/U)$, and the bound $B/8$ contributes $O(\sqrt M\log(2M))$. For curvature, put $\theta=1/(2R+2)$; the contribution is at most
\[C_{R,c}N^\theta\sum_{a\le\sqrt M}a^\theta\int_a^{M/a}B^{-1+\theta}dB\ll_{R,c}N^\theta M^{1/2+\theta}.\]
The endpoint errors in these integral comparisons cost at most one further block per row. Together with the exceptions and final fragments, this proves \eqref{eq:block-count}.

For the lower bound it suffices to consider the $\gg M$ pairs in the rectangle
\[\frac{\sqrt M}{4}\le a\le\frac{\sqrt M}{2},\quad \sqrt M\le b\le\frac{3\sqrt M}{2}.\]
Any admissible block meeting this rectangle has $B\asymp\sqrt M$, by $K\le B/8$, and hence length at most
\[\min\{U,C_{R,c}N^{-\theta}M^{R/(2R+2)}\}.\]
For $M\ge N^{1/R}$ this also bounds the length of an individual summand, after increasing the constant. Dividing the number of pairs by the maximum length proves \eqref{eq:curvature-lower}, with bounded $M$ absorbed by the implied constant.
\end{proof}

This count explains why we retained four terms of the nonconstant phase. With $H=N^{h+o(1)}$ and $M=N^{1-2h+o(1)}$, the curvature term has exponent
\[\frac{1+(R+2)(1-2h)}{2R+2};\]
keeping it at most $h$ requires $h\ge(R+3)/(4R+6)$, so even before paying for a moment table, $R=3$ cannot improve on $1/3$ by this scheme.

For quartic phases, the unrestricted tables for dyadic $K\le U$ cost $\widetilde O(U^{10})$ in total (Lemma \ref{lem:table}), and balancing this against the cap contribution $\widetilde O(M/U)$ selects $U\asymp M^{1/11}$. With $R=4$ and $c=2^{-20}$ in \eqref{eq:block-count}, and the diagonal terms of \eqref{eq:SformTargetSum} included, the cost is
\begin{equation}
  \widetilde O\left(M^{10/11}+N^{1/10}M^{3/5}+N^{1/4}+\sqrt M\right).
  \label{eq:unrestricted-cost}
\end{equation}
Taking $H=N^{10/31+o(1)}$ gives $M=N^{11/31+o(1)}$, so the first term balances the smoothing cost, while the curvature term has exponent $97/310<10/31$. The unrestricted table thus gives an $N^{10/31+o(1)}$ algorithm.

The improvement to $N^{9/28+o(1)}$ comes from the fact that the fourth Taylor coefficient becomes small when $aB$ is large, leaving many fourth residues unused. We first construct the table in essentially its own size -- completing the $N^{10/31+o(1)}$ algorithm -- in a form that allows this restriction; the actual construction to exploit it is saved for Section \ref{sec:sec6}.

\section{Precomputing \texorpdfstring{$T_r(K;q_1,q_2,q_3,q_4)$}{the moment table}}
\label{sec:sec5}

The table can be constructed one Fourier transform at a time. Fixing $q_2,q_3,q_4$ and $r$ leaves $K$ moments indexed by $q_1$; their simultaneous evaluation costs $O(K\log(2K))$ operations by Bluestein's algorithm \cite{bluestein1970}. Thus the construction takes essentially the size of the table.

\begin{lemma}\label{lem:table}
For integers $K\ge1$, $d\ge0$, and $p\ge1$, all moments
\[T_r(K;q_1,q_2,q_3,q_4)=\sum_{k=0}^{K-1}(k/K)^re\left(\sum_{j=1}^4q_j(k/K)^j\right),\quad 0\le r\le d,\quad 0\le q_j<K^j,\]
with $0^0=1$, can be computed to error $2^{-p}$ in
\[K^{10}(d+p+\log(2K))^{O(1)}\]
bit operations and bits of storage.
\end{lemma}

\begin{proof}
For fixed $q_2,q_3,q_4,r$, form the vector
\[v_k=(k/K)^re\left(q_2(k/K)^2+q_3(k/K)^3+q_4(k/K)^4\right).\]
Its positive-sign Fourier transform is exactly the required set of moments:
\[T_r(K;q_1,q_2,q_3,q_4)=\sum_{k=0}^{K-1}v_ke(q_1k/K),\quad 0\le q_1<K.\]
There are $(d+1)K^9$ such transforms, each requiring $K$ samples and $O(K\log(2K))$ arithmetic operations. To form the samples, reduce the rational phases modulo one exactly and evaluate the factors $(k/K)^r$, which have $O(d\log(2K))$ bits. Since $|v_k|\le1$, $p+O(\log(2K))$ bits of working precision suffice for the transforms, giving the stated bit cost.
\end{proof}

In particular, these transforms never mix fourth coordinates. If we need only an explicit set $\mathcal Q\subseteq\{0,\ldots,K^4-1\}$ of such, we construct just the corresponding transforms, at cost
\[K^6|\mathcal Q|(d+p+\log(2K))^{O(1)}.\]
If we can therefore restrict the fourth coordinate so as to reduce the size of the table necessary for the computation of $S(a,B,K)$, construction time falls by the same factor.

\section{Dyadic shells}
\label{sec:sec6}

To restrict the fourth coordinate, group the pairs $a\le b$, $ab\le M$ into dyadic product ranges $X\le ab<2X$, where $X=1,2,4,\ldots$. Each such shell will have its own table and length cap $U_X$. For a block beginning at $B$, the fourth Taylor coefficient is
\begin{equation}
  \alpha_4=-\frac5{64}\sqrt{Na}B^{-7/2}K^4=-\frac5{64}\sqrt N(aB)^{-3/2}(a/B)^2K^4.
  \label{eq:shell-fourth-coefficient}
\end{equation}
The factor $(a/B)^2\le1$ gives a bound uniform over all rows of the shell:
\[|\alpha_4|\le(5/64)\sqrt N X^{-3/2}K^4.\]
Thus for $X$ large compared with $N^{1/3}$, the fourth coefficient occupies only a small part of the full residue range of length $K^4$.

To count the required residues, put
\[C_X=\min(1,2^{j_X}),\quad j_X=\min\{j\in\mathbb Z:2^{2j}X^3\ge N\}.\]
When $C_X<1$ we have $\sqrt N X^{-3/2}\le C_X<2\sqrt N X^{-3/2}$. Since $\alpha_4$ is negative and $|q_4-\alpha_4|\le1$, all required fourth residues belong to
\[\mathcal Q_{X,K}=
\begin{cases}
\{0,\ldots,K^4-1\},&C_X=1,\\
\{-\lceil C_XK^4\rceil-1,\ldots,1\}\pmod{K^4},&C_X<1.
\end{cases}\]
Keeping each residue only once gives $|\mathcal Q_{X,K}|=O(1+C_XK^4)$, so at the degrees and precision required by Theorem \ref{thm:block}, the restricted table costs $\widetilde O(K^6+C_XK^{10})$. How, then, should we choose the block length caps $U_X$ to match the smaller table? It would contribute $\widetilde O(X/U_X)$ blocks, while the tables for all dyadic $K\le U_X$ together cost
\[\widetilde O(U_X^6+C_XU_X^{10}).\]
Balancing the second term against the block count, take $U_X$ to be the largest dyadic integer with $C_XU_X^{11}\le X$, so $U_X\asymp(X/C_X)^{1/11}$. The extra term $U_X^6$, arising from rounding the fourth coefficient, is also covered by $X/U_X$ in the range $M\le N^{2/5}$ needed below: if $C_X=1$, then $U_X^7\le U_X^{11}\le X$, and otherwise
\[\frac{U_X^7}{X}\le X^{-4/11}C_X^{-7/11}\le N^{-7/22}X^{13/22}\le1.\]
Thus each shell's table costs no more than its cap contribution to the block count. It remains to sum these contributions and check the shorter blocks forced by curvature.

\begin{theorem}\label{thm:shells}
Suppose $2\le M\le N^{2/5}$. For any fixed $A>0$, the finite sum in \eqref{eq:Sform2} can be evaluated to error $N^{-A}$ in
\begin{equation}
  \widetilde O\left(N^{1/22}M^{17/22}+N^{10/33}+N^{1/10}M^{3/5}+N^{1/4}+\sqrt M\right)
  \label{eq:shell-total-cost}
\end{equation}
bit operations and bits of storage.
\end{theorem}

\begin{proof}
In each shell use the greedy partition of Lemma \ref{lem:block-count}, with cap $U_X$ and blocks stopped at the shell boundary. The cap contributes $O(X\log(2X)/U_X)$ blocks; row endpoints and the restriction $8K\le B$ contribute $O(\sqrt X\log(2X))$. For curvature, the integral comparison now has upper endpoint $2X/a$, giving
\[N^{1/10}\sum_{a\le\sqrt{2X}}a^{1/10}\int_a^{2X/a}B^{-9/10}\,dB\ll N^{1/10}X^{3/5}.\]
Across all shells, the individually evaluated exceptions still satisfy $B<8$ or $B^8<2^{20}N$, and number $O(N^{1/4})$. Summing the row and curvature terms geometrically, and including the diagonal in \eqref{eq:SformTargetSum}, leaves the total cost
\begin{align*}
  \widetilde O\left(\sum_X\frac X{U_X}+N^{1/10}M^{3/5}+N^{1/4}+\sqrt M\right),
\end{align*}
with table construction already covered by the first term.

Only the sum over caps remains. The shells with $C_X=1$ have $X<(4N)^{1/3}$ and $X/U_X\asymp X^{10/11}$, so together they cost $O(N^{10/33})$. On the larger shells the saving in the fourth coordinate applies:
\[\frac X{U_X}\asymp C_X^{1/11}X^{10/11}\asymp N^{1/22}X^{17/22}.\]
Their costs sum geometrically to $O(N^{1/22}M^{17/22})$, completing \eqref{eq:shell-total-cost}.

Taking error $N^{-A-3}$ per evaluation suffices, since the total absolute multiplicity is at most $D(M)+2\lfloor\sqrt M\rfloor<N^3$. The degrees and precision remain polylogarithmic by Theorem \ref{thm:block}. For each shell, form the block polynomials first to determine their maximum degree and coefficient bound, then construct the corresponding tables. Sorting requests and table entries by length and key allows sequential contractions with only logarithmic overhead. Storing one shell's requests, tables, and sorting buffers at a time proves the same bound for space.
\end{proof}

The largest shells contribute $N^{1/22}M^{17/22}$ through their tables and block queries. Ignoring subpolynomial factors, the balance with the smoothing cost is
\[H\asymp N^{1/22}(N/H^2)^{17/22}=N^{9/11}H^{-17/11},\]
which selects $H=N^{9/28+o(1)}$ and $M=N^{5/14+o(1)}$. At this choice the other exponents of \eqref{eq:shell-total-cost} are $10/33$, $11/35$, $1/4$, and $5/28$, and the exponent of the first term, corresponding to the shared tables and the number of blocks that use them, is $9/28$.

\section*{References}
	
{
\setlength{\emergencystretch}{1em}
\printbibliography[heading=none] 
}

\end{document}